\documentclass[12 pt]{article}

\usepackage{graphicx}
\usepackage{amsmath, amsthm}
\usepackage{amsfonts}
\usepackage{amssymb}
\usepackage{url}
\usepackage{verbatim}
\usepackage[mathlines]{lineno}
\usepackage[dvipsnames]{xcolor}
\usepackage{color}
\usepackage{hyperref}
\usepackage{fontenc}

\newtheorem{corollary}{Corollary}

\newtheorem{theorem}{Theorem}
\newtheorem{lemma}{Lemma}

\newtheorem{example}{Example}
\newtheorem{conjecture}{Conjecture}

\newtheorem{remark}{Remark}
\DeclareMathOperator{\dist}{dist}

\title{Lower Bounds for the Exponential Domination Number of $C_m \times C_n$}
\author{Chassidy Bozeman$^1$ \and Joshua Carlson$^1$ \and Michael Dairyko$^1$ \and Derek Young$^1$ \and Michael Young$^{1,2}$}
\date{April 22, 2016}

\begin{document}
\maketitle
\footnotetext[1]{Department of Mathematics, Iowa State University, Ames, IA 50011, USA. \{cbozeman, jmsdg7, mdairyko, ddyoung, myoung\}@iastate.edu}
\footnotetext[2]{Corresponding author.}

%----------------------------------------------------------------------------------------------------------------------------------------------------------------------------
\begin{abstract}
\noindent
A vertex $v$ in a porous exponential dominating set assigns weight $\left(\tfrac{1}{2}\right)^{\dist(v,u)}$ to vertex $u$. A porous exponential dominating set of a graph $G$ is a subset of $V(G)$ such that every vertex in $V(G)$ has been assigned a sum weight of at least 1. In this paper the porous exponential dominating number, denoted by $\gamma_e^*(G)$, for the graph $G = C_m \times C_n$ is discussed. Anderson et. al. \cite{anderson} proved that $\frac{mn}{15.875}\le \gamma_e^*(C_m \times C_n) \le \frac{mn}{13}$ and conjectured that $\frac{mn}{13}$ is also the asymptotic lower bound. We use a linear programing approach to sharpen the lower bound to $\frac{mn}{13.7619 + \epsilon(m,n)}$.
\\*[2mm]
\noindent

\end{abstract}

\section{Introduction}
%Given a graph $G = (V,E),$ we define a {\em weight function} to be a function $w: V(G)\times V(G) \to \mathbb{R},$ such that for $a,b \in V(G),$  $w(a,b)$ represents the weight that $a$ assigns to $b$ in $G.$ 

Given a graph $G$, a \emph{weight function} of $G$ is a function $w:V(G)\times V(G)\rightarrow \mathbb{R}$. For $u,v \in V(G)$, we say that $u$ assigns weight $w(u,v)$ to $v$. For a set $S \subseteq V(G)$ we denote $$w(S,v)  := \sum_{s \in S} w(s,v).$$ Similarly, $$w(v,S)  := \sum_{s \in S} w(v,s).$$ For two weight functions of $G$, $w$ and $w'$, we say $w' \leq w$, if $w'(u,v) \leq w(u,v)$ for all $u,v \in V(G)$.

Let $D \subseteq V(G)$ and $w$ be a weight function of $G$. The pair $(D,w)$ \emph{dominates} the graph $G$, if for all $v \in V(G)$, $w(D,v) \ge 1$. 

The standard definition of domination is where $w(u,v)$ is 1 if $v$ is in the closed neighborhood of $u$, and 0 if it is not. This type of domination has been widely studied (see \cite{GPRT}, \cite{Haynes}). Another well-studied type of domination is \emph{total domination}, in which $w(u,v)$ is 1 if $v$ is in the neighborhood of $u$ and 0 if it is not (see \cite{G}, \cite{H}). There is also \emph{$k$-domination} and \emph{$k$-distance domination} (see \cite{CFHV}, \cite{HMV}). In $k$-domination $w(u,v)$ is $\frac{1}{k}$ if $v$ is in the neighborhood of $u$, 1 if $u=v$, and 0 otherwise. In $k$-distance domination $w(u,v)$ is 1 if the distance from $u$ to $v$ is at most $k$, and 0 otherwise. For \emph{exponential domination}, $w(u,v) = (\frac{1}{2})^{\dist(u,v) -1}$, where $\dist(u,v)$ represents the length of the shortest path from $u$ to $v$.

A \emph{porous exponential dominating set} of a graph $G$ is a set $D \subseteq V(G)$ such that $(D,w)$ dominates $G$ when $w(u,v) = (\frac{1}{2})^{\dist(u,v) -1}$. The \emph{exponential domination number} of $G$, denoted by $\gamma_e^*(G),$ is the cardinality of a minimum exponential dominating set. This type of exponential domination has also been referred to as \emph{porous exponential domination}. Some work has been done in \emph{non-porous exponential domination}, where $\dist(u,v)$ represents the length of the shortest path from $u$ to $v$ that does not have any internal vertices that are in the dominating set (see \cite{dankel}). For the sake of simplicity, we refer to porous exponential domination as exponential domination.

It is described in \cite{dankel} that applications of exponential domination relate to the passing of information, and specifically models how information can be spread from a speaker through a crowd. Thus, the exponential domination number represents the minimum number of speakers needed to successfully convey a message to every individual within a crowd.

Within exponential domination, there has been research on the exponential domination number of $C_m \times C_n,$ where $\times$ is the Cartesian product. Anderson et. al. \cite{anderson}, found lower and upper bounds for $\gamma_e^*(C_m \times C_n).$  The following theorem shows a sharp upper bound for the exponential domination number of the graph $C_{13m} \times C_{13n}.$

\begin{theorem}\cite{anderson} \label{13} 
For all $m$ and $n,$ $$\frac{\gamma_e^*(C_{13m} \times C_{13n})}{(13m)(13n)} \le \frac{1}{13}.$$ 
\end{theorem}

The proof of Theorem \ref{13} is constructive. An exponential dominating set is created by choosing a vertex from each row and column of a $13 \times 13$ grid and then periodically tiling $C_{13m} \times C_{13n}$ with the grid and selecting all the corresponding vertices. This argument was extended to $C_m \times C_n,$ for $m,n$ arbitrarily large.

\begin{theorem}\cite{anderson} \label{lim13}
$$ \lim_{m,n \to \infty} \frac{\gamma_e^*(C_m \times C_n)}{mn} \le \frac{1}{13}.$$ 
\end{theorem}

A lower bound can be attained in the following way: Observe that when $m$ and $n$ are large enough, for each $v \in V(C_m \times C_n)$ there exists $4i$ vertices $u \in V(C_m \times C_n)$ such that $\dist(v,u) = i$, when $i$ is a positive integer. So $w(v,V(C_m \times C_n)) \le 2+\displaystyle\sum_{i=1}^{\infty} 4i 2^{1-i} = 18$. This implies $$ \frac{1}{18} \le \frac{\gamma_e^*(G)}{mn}.$$

However this bound can be improved to $\frac{1}{17}$ by adjusting the weight function so that $v$ assigns weight 1 to itself, resulting in $w(v,V(C_m \times C_n)) \le 17$. A better lower bound was attained in \cite{anderson} by showing that the weight function could be adjusted so that each vertex assigns 2.125 less than in the original weight function.

\begin{theorem} \cite{anderson}
For all $m,n \ge 3$, $$\frac{1}{15.875} < \frac{\gamma_e^*(G)}{mn}.$$
\end{theorem} 

The above theorems led to the following conjecture.

\begin{conjecture}\cite{anderson} \label{conj} 
For all $m$ and $n,$ $$\frac{1}{13} \le \frac{\gamma_e^*(C_{m} \times C_{n})}{mn}.$$
\end{conjecture}

In this paper, a better lower bound for $\frac{\gamma_e^*(C_{m} \times C_{n})}{mn}$ is produced. In section \ref{sec:lp}, we use linear programming to minimize how much weight is necessary for each vertex in an exponential dominating set to assign. This leads to improved lower bounds in section \ref{sec:res}. For the remainder of the paper, we refer to $w$ as the weight function $$w(u,v) := \left(\dfrac{1}{2}\right)^{\dist(u,v) -1}.$$

%----------------------------------------------------------------------------------------------------------------------------------------------------------------------------
\section{Linear Program} \label{sec:lp}
For the rest of the paper, let $G = C_m \times C_n$ and let $D = \{d_1, d_2, \ldots, d_{|D|} \}$ be an exponential dominating set of $G.$ Given an odd positive integer $r$, let $G_v$ be the subgraph of $G$ that is an $r \times r$ grid centered vertex $v \in V(G)$, with $V(G_v) = \{ v_1, v_2, \ldots, v_{r^2} \}$. Let $I_v$ be the set of interior vertices of $G_v$.

\indent In this section, a linear program is created where the sum of the weights assigned to the vertices in $I_{v}$ is minimized. The minimum value attained is of the form $|I_{v}| + k$, where $0<k$. A new weight function is then created, which still dominates with $D$ and has $v$ assigning $k$ less weight than before. A sequence of weight functions will be constructed recursively. \\

\begin{subsection}{The Grid}
\indent First we strategically partition $V(G)$. For each $v_i \in V(G_v)$, define $S_i$ to be the set of vertices $w \in V(G)$ such that the distance between $v_i$ and $w$ is less than the distance between $w$ and any other vertex in $G_v$. Notice that $S_i = \{v_i\}$, if $v_i \in I_v$. For $1 \le i \le r^2$, let $x_i = w(S_i, v_i)$. Therefore, if $1 \le i,j \le r^2$, then $w(S_i, v_j) = x_i \left(\tfrac{1}{2} \right) ^ {\dist(v_i,v_j)}.$ We define $\Gamma = V(G) \setminus  \bigcup_{i=1}^{r^2} S_i$ and for $1 \le j \le r^2,$ let $\epsilon_j = w(\Gamma,v_j)$. Thus, 

$$w(D,v_j) \le \displaystyle\sum_{i=1}^{r^2} w(S_i, v_j) + \epsilon_j = \displaystyle\sum_{i=1}^{r^2}x_i \left(\frac{1}{2} \right) ^ {\dist(v_i,v_j)} + \epsilon_j.$$ 

 Observe that $|V(\Gamma)| \le m+n-1$ and $\dist(\Gamma, V(G_v)) \rightarrow \infty$ as $m,n \rightarrow \infty$. Thus $0 \le \epsilon_j \le (m+n-1)\left( \frac{1}{2} \right)^{\dist(\Gamma, V(G_v)) - 1}$ for each $1 \le j \le r^2$. Therefore assuming that $\epsilon = \displaystyle\sum_{j=1}^{r^2} \epsilon_j,$
 
\[
\epsilon  \le \displaystyle\sum_{j=1}^{r^2} (m+n-1)\left( \frac{1}{2} \right)^{\dist(\Gamma, V(G_v)) - 1} \\
\le r^2(m+n-1)\left( \frac{1}{2} \right)^{\dist(\Gamma, V(G_v)) - 1},
\]
which means $\epsilon \to 0$ \text{ as } $m,n \to \infty$.

\begin{example}{\rm 
Consider $G =C_6 \times C_8$ as shown in Figure \ref{gridexample}. For the simplicity of the figure, we remove the edges of $G.$ Choose $r=3$ and construct $G_{v_5}$ with $V(G_{v_5}) = \{v_1, v_2, \ldots , v_9\}$. We then label the corresponding sets $S_1, S_2, \ldots, S_9, \Gamma.$ For instance, observe that $S_3$ consists of all vertices in $G$ whose distance to $v_3$ is smaller than their distance to any other vertex of $G_{v_5}.$
\begin{figure}[htb!]
    \centering
\includegraphics[width = 10cm]{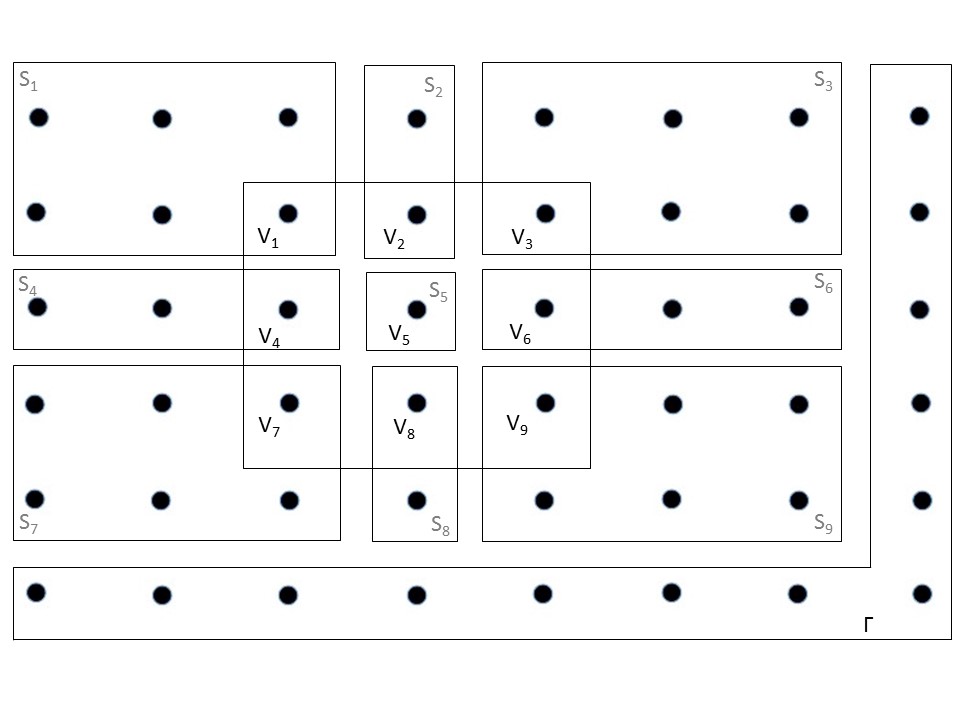}
    \caption{$C_6 \times C_8$ with edges removed}
    \label{gridexample}
\end{figure}}
\end{example}
\end{subsection}

\begin{subsection}{The Program}

%We actually prove that if $(D,w_{j-1})$ dominates $G$ and $w_{j-1} \le w$, then there exists a weight function $w_j$ such that $1)$, $2)$, and $3)$ are satisfied and the only vertex that assigns different weight to $V(G)$ is $d_j$.

Lemma \ref{bound} below proves how to get a lower bound for the exponential domination number of a graph $G$, given that each vertex in the dominating set assigns more weight to $V(G)$ than needed. 

\begin{lemma}\label{bound}
Let $D = \{d_1, d_2, \ldots, d_{|D|} \}$ be an exponential dominating set of $G$ and $\rho \in \mathbb{R}$ such that $w(d_j,V(G)) \le \rho$ for all $j$. If there exists a sequence of weight functions $\{w_j\}_{j=0}^{|D|}$, where $w = w_0$, and the following conditions are satisfied for $1 \le j \le |D|$,

\begin{enumerate}
\item [1)]
$w_j < w_{j-1}$, 
\item [2)]f
$(D,w_j)$ dominates $G$, and
\item [3)]
there exist $k \in \mathbb{R}$ such that $0 < k \le w_{j-1}(d_j,V(G)) - w_j(d_j,V(G))$,
\end{enumerate}

then $$\frac{1}{\rho-k} < \frac{|D|}{|V(G)|}.$$
\end{lemma}

\begin{proof}
Let $\{w_j\}_{j=0}^{|D|}$ be such a sequence of weight functions for the exponential dominating set $D$. Conditions $1)$ and $3)$ imply that $ k < w_0(d_j,V(G)) - w_{|D|}(d_j,V(G))$ for all $d_j \in D$. Therefore, $$k|D| < \displaystyle\sum_{j=1}^{|D|} [w_0(d_j,V(G)) - w_{|D|}(d_j,V(G))].$$ Since condition $2)$ gives $1 \le w_{|D|}(D,v)$ for all $v \in V(G)$, then 
$$|V(G)| \le \displaystyle\sum_{j=1}^{|D|} w_{|D|}(d_j,V(G)).$$ Combining these inequalities gives,
\begin{align*}
k|D| + |V(G)|  & <  \sum_{j=1}^{|D|} w_{0}(d_j,V(G))\\
&\le \sum_{j=1}^{|D|} 18 = 18 |D|.
\end{align*}
This implies that $$\frac{1}{18 - k} <  \dfrac{|D|}{|V(G)|}.$$  
\end{proof}

We now construct a recursive set of weight functions that satisfy the conditions of Lemma \ref{bound} for some $k.$ Let $d_j \in D$ and $w_{j-1}$ be a weight function such that $(D,w_{j-1})$ dominates $G$. Let $G_d$ be the $r \times r$ grid $G_{d_j}$ and $I = I_{d_j}$. Recall that $x_i = w(S_i,v_i).$ Let $A$ be the $r \times r$ matrix such that $[A]_{ij} = \left(\frac{1}{2} \right) ^ {\dist(v_i,v_j)}.$ Let $\vec{x} = [x_1, x_2, \ldots, x_{r^2}]^{\rm \bf T}$ and $\vec{w} = [w(D,v_1), w(D,v_2), \ldots, w(D,v_{r^2})]^{\rm \bf T}.$ Thus, $\vec{w} \le A \vec{x}.$ In fact, if $w_0 < w$, then $\vec{w_0} < A {\bf x}.$

Let $c$ be the real-valued vector such that 
$$c^{\bf T} \vec{x} = \displaystyle\sum_{v_i \in I} w_{j-1}(D,v_i).$$ 
The objective function in the linear program will be $c^{\bf T} {\bf x}$, where ${\bf x}$ is a vector of $r^2$ variables. Since $(D,w_{j-1})$ dominates $G$, ${\bf 1} \le \vec{w}_{j-1}$, where ${\bf 1}$ is the all 1s vector. Therefore, ${\bf 1} \le A\vec{x}$; hence, ${\bf 1} \le A {\bf x}$ is a constraint. Let $b$ be the real-valued vector whose $i$th entry is $1+\left(\frac{1}{2} \right) ^ {\dist(v_i,d_j)} + \epsilon_j$ if $v_i \in I$ and $18$ otherwise. The constraint $A{\bf x} \le b$ will be added to ensure that for each vertex in $I$ the weight assigned from $d_j$ can be decreased by the appropriate amount. Consider the following linear program:
\begin{equation*}
\begin{array}{rrcll}
\min &{c}^{\bf T} {\bf x} &&\\
{\rm s.t.} & A{\bf x} &\ge& {\bf 1}\\
& A{ \bf x} &\le& b\\
&{\bf x}& \ge & {\bf 0}. &\\
\end{array}
\end{equation*}

Define ${\bf x}^*$ to be an optimal solution to the linear program and ${\bf x}_{\min}$ to be the value attained. Obviously, $|I| + \epsilon < {\bf x}_{\min}$, so $0 < k = {\bf x}_{\min} - \epsilon - |I|$. For each $i$ with $v_i \in I$, let $$y_i = \displaystyle\sum_{s=1}^{r^2} {\bf x}_i \left(\frac{1}{2} \right) ^ {\dist(v_i,v_s)} - \epsilon_i -1.$$ Thus, $0 \le y_i \le \left(\frac{1}{2} \right) ^ {\dist(v_i,d_j)}$ and $\displaystyle\sum_{v_i \in I} y_i = k.$ 
\begin{remark}{\rm
Note that the weights function $\{w_j\}_{j=0}^{|D|}$ satisfy conditions 1), 2), and 3) of Lemma \ref{bound}. Clearly $w_j < w_{j-1}$, so $1)$ is satisfied. For each $v \in V(G) \setminus I$,  $1 \le w_{j-1}(D,v) = w_j(D,v)$. For each $v_i \in I$, $w_j(D,v) = w_{j-1}(D,v) - k = 1 + \epsilon_i$. This implies $(D, w_j)$ dominates $G$ so $2)$ is satisfied. Lastly, $w_j(d_j,V(G)) = w_{j-1}(d_j,V(G)) - k$, so $3)$ is satisfied.
}
\end{remark}

\end{subsection}

\section{Main Results} \label{sec:res}
In this section, we use Lemma \ref{bound}  and the weight functions $\{w_j\}_{j=0}^{|D|}$ constructed in Section \ref{sec:lp} to attain a lower bound for the exponential domination number of $C_m \times C_n$.

\begin{theorem}\label{main}
For all $m,n \ge 13$, $$\frac{1}{13.7619 +\epsilon} \le \frac{\gamma_e^*(C_m\times C_n)}{mn},$$ where $\epsilon\rightarrow 0 \text{ as } m,n \rightarrow \infty.$ Moreover, $\epsilon = 0$ when $m$ and $n$ are both odd.
\end{theorem}

\begin{proof}
Let $D$ be a minimum exponential dominating set. For each $v \in D$, let $G_v$ be the $13 \times 13$ grid centered at $v$. Recall that $w(v,V(G)) \le 18$ for all $v \in D$. The solution to the corresponding linear program is $x_{min} = 125.2381080608$. Therefore, it follows that $k = 125.2381080608 - \epsilon - 121 = 4.2381080608$, so $\frac{mn}{13.7618919392 +\epsilon} \le \gamma_e^*(C_m\times C_n)$ by Lemma \ref{bound}. 
\end{proof}

The linear program created in Section \ref{sec:lp} can be constructed in the form of a mixed integer linear program by adding the constraints ${\bf x}_i = 0$ or 2, when $v_i \in I_v$. Then the attained $k$ is $10.94 + \epsilon_v$ by choosing a $9 \times 9$ grid as $G_v$. However, the weight function can only be adjusted at a vertex $v \in D$, such that no vertices in $D \cap I_v$ have been adjusted. Rather than using the linear program for all the vertices in $D$, we will use it for the vertices in $D$ that are relatively close together and use the mixed integer linear program for those vertices in $D$ that are not close to the other vertices of $D$.

\begin{theorem} \label{genbound}
Let $D$ be an exponential dominating set of $C_m \times C_n$ and $\alpha |D|$ be the number of vertices in $D$ that are not within a $7 \times 7$ grid of any other vertex in $D$. Then $$\frac{1}{13.7619 - 2.8218 \alpha - \epsilon} \le \frac{\gamma(C_m \times C_n)}{mn},$$ where $\epsilon \to 0$ as $m,n \to \infty$.
\end{theorem}

\begin{proof}
Let $D'$ be the set of vertices that are not within a $7 \times 7$ grid of any other vertex in $D$; so $|D'| = \alpha|D|$. Choose $r=9$ in $G_v$ and let $b'$  be the real-valued vector whose $i$th entry is $0$ if $v_i \in I_v$ and $4$ otherwise. By taking geometric sums, it is easy to see that $x_i \le 4$, for all $i$.

The linear program 
\begin{equation*}
\begin{array}{rrcl}
\min &{c}^{\bf T} {\bf x} &&\\
{\rm s.t.} & A{\bf x} &\ge& {\bf 1}\\
& A{ \bf x} &\le& b\\
&{\bf x}& \le &b'   \\
&{\bf x}& \ge & {\bf 0}. \\
\end{array}
\end{equation*}
will attain a minimum of 56.06. So each vertex in $D'$ can be adjusted by $56.06 -\epsilon_2 - 49 =7.06 - \epsilon_2$ to $10.94 + \epsilon_2$, for some $\epsilon_2 \ge 0$. As before, the vertices of $D \setminus D'$ can be adjusted to $13.7618919392 + \epsilon_1$, for some $\epsilon_1 \ge 0$. So $mn \le (1- \alpha)|D|(13.7619 + \epsilon_1) + \alpha|D|(10.94 + \epsilon_2)$, which implies $mn \le |D| (13.7619 - 2.8218\alpha + \epsilon)$.
\end{proof}

Corollary \ref{alphabound} is a direct result of combining Theorems \ref{13} and \ref{genbound}.

\begin{corollary} \label{alphabound}
Let $D$ be an exponential dominating set of $C_m \times C_n$. For $m$ and $n$ large enough, the number of vertices in $D$ that are not within a $7 \times 7$ grid of any other vertex in $D$ is at most $.27 |D|.$
\end{corollary}

%\newpage

%\begin{theorem} 
%Let $G=C_m\times C_n$. Then, $$\frac{mn}{13.76 +\epsilon(m,n)} \le \gamma_e^*(G),$$ where $\epsilon\rightarrow 0 \text{ as } m,n \rightarrow \infty.$ Moreover, $\epsilon(m,n) = 0$ when $m$ and $n$ are both odd.
%\end{theorem} 

%\begin{proof}
%Using $\vec{x}, \vec{y}$ as defined above, let $$\hat{f}(\vec{x}):=\sum\limits_{u\in I}w(D,u).$$ and   $$f(\vec{x},\vec{y}) := \sum\limits_{u\in I} w(D,u).$$ 

%For any vertex  in $I$, the weight that it is assigned from a vertex in \textcolor{red}{$S_i$} is at most $(\frac{1}{2})^ {\frac{ \min(m,n)}{2}}$. So $$f(\vec{x},\vec{y}) \le \hat{f}(\vec{x})+ |I|(m+n-1) (\frac{1}{2})^ {\frac{\min(m,n)}{2}}.$$ It follows that $$\lim\limits_{m,n \rightarrow \infty} f(\vec{x},\vec{y}) = \lim\limits_{m,n\rightarrow \infty}\hat{f}(\vec{x}).$$

%Define $x_{min}, \hat{x}_{min}$ to be the minimum values for $f$ and $\hat{f}$ respectfully, and define $\epsilon(m,n) = \hat{x}_{min} - x_{min}.$

%\textcolor{red}{FINISH!!!}
%\end{proof}

%----------------------------------------------------------------------------------------------------------------------------------------------------------------------------

\end{document}